\documentclass{amsart}
\usepackage{graphicx}
\newtheorem{thm}{Theorem}[section]

\newtheorem{lem}[thm]{Lemma}

\theoremstyle{definition}
\newtheorem{defn}[thm]{Definition}
\theoremstyle{remark}
\newtheorem{rem}[thm]{Remark}
\numberwithin{equation}{section}

\begin{document}

\title[Commutators of multilinear singular integral operators]{Commutators of multilinear
singular integral operators on non-homogeneous metric measure spaces}%
\author{Rulong Xie, Huajun Gong$^*$ and Xiaoyao Zhou }
 \footnotetext {* Corresponding author}
\address{School of Mathematical Sciences, University of
Science and Technology of China, Hefei 230026, China}
\email{xierl@mail.ustc.edu.cn;huajun84@hotmail.com;zhouxiaoyaodeyouxian@126.com}%

\begin{abstract}
Let $(X,d,\mu)$ be a metric measure space satisfying both the
geometrically doubling and the upper doubling measure conditions,
which is called non-homogeneous metric measure space. In this
paper, via a sharp maximal operator, the boundedness of commutators
generated by multilinear singular integral with $RBMO(\mu)$ function
on non-homogeneous metric measure spaces in $m$-multiple Lebesgue spaces
is obtained.
\end{abstract}
\maketitle

\section{Introduction}

It is well known that the standard singular integral theory is
constructed with the assumption of spaces satisfying the doubling
measures condition. We recall that $\mu$ is said to satisfy the
doubling condition if there exists a constant $C>0$ such that
$\mu(B(x,2r))\leq C\mu(B(x,r))$ for all $x\in \rm{supp} \mu$ and
$r>0$. A metric measure space $(X,d,\mu)$ equipped with a
non-negative doubling  measure $\mu$ is called a space of
homogeneous type. In case of non-doubling measures, a non-negative
measure $\mu$ only need to satisfy the polynomial growth condition,
i.e., for all $x \in \mathbb{R}^{n}$ and $r>0$, there exist a constant $C_0>0$ and
$k\in(0, n]$ such that,
\begin{equation}\
\mu(B(x, r)) \leq C_{0}r^{k},
\end{equation}
where $B(x, r)= \{y \in \mathbb{R}^{n}: |y-x| < r\}$. This breakthrough brings
rapid development in harmonic analysis (see [2,5,8-9,17-22]). And
the analysis on non-doubling measures has important applications in
solving the long-standing open Painlev$\acute{e}$'s problem (see
[18]).

   However, as stated by Hyt\"{o}nen in [11], the measure
satisfying (1) does not include the doubling measure as special cases.
To solve this problem, a kind of metric measure space $(X,d,\mu)$,
which is called non-homogeneous metric measures space, satisfying
geometrically doubling and the upper doubling measure condition (see
Definition 1.1 and 1.2) is introduced by Hyt\"{o}nen in [11]. The
highlight of this kind of spaces is that it includes both the
homogeneous spaces and metric spaces with polynomial growth measures
as special cases. From then on, some results paralled to homogeneous
spaces and non-doubling measures space are obtained (see
[1,5,11-16] and the references therein). Hyt\"{o}nen et al. in [14]
and Bui and Duong in [1] independently introduced the atomic Hardy
space $H^{1}(\mu)$ and proved that the dual space of $H^{1}(\mu)$ is
$RBMO(\mu)$. In [1], the authors also proved that
Calder\'{o}n-Zygmund operator and commutators of
Calder\'{o}n-Zygmund operators and RBMO functions are bounded in
$L^{p}(\mu)$ for $1<p<\infty$. Recently, some equivalent
characterizations are established by Liu et al. in [16] for the
boundedness of Carder¡äon-Zygmund operators on $L^{p}(\mu)$ for
$1<p<\infty$. In [4], Fu et al. established boundedness of
multilinear commutators of Calder\'{o}n-Zygmund operators on Orlicz
spaces on non-homogeneous spaces.

   On the other hand, the theory on multilinear singular integral
operators has been considered by some researchers. In [3], Coifman
and Meyers firstly established the theory of bilinear
Calder\'{o}n-Zygmund operators. Later, Gorafakos and Torres [6-7]
established the boundedness of multilinear singular integral on the
product Lebesgue spaces and Hardy spaces. The properties of
multilinear singular integrals and commutators on non-doubling
measures spaces $(\mathbb{R}^{n},\mu)$ were established by Xu in [21-22].
Weighted norm inequalities for multilinear Calder\'{o}n-Zygmund
operators on non-homogeneous metric measure spaces were also
constructed in [10].

In the setting of non-homogeneous metric measure spaces, it is
natural to ask whether commutators of multilinear singular integral
operators is also bounded in $m$-multiple Lebesgue spaces. This paper
will give an affirmative answer to this question. In this paper,
commutators generated by multilinear singular integrals with
$RBMO(\mu)$ function on non-homogeneous metric spaces is introduced
firstly. And we will  prove that it is bounded in $m$-multiple Lebesgue
spaces on non-homogeneous metric spaces, provided that multilinear
singular integrals is bounded from $m$-multiple $L^1(\mu)\times \ldots
\times L^1(\mu)$ to $L^{1/m,\infty}(\mu)$, where $L^{p}(\mu)$ and
$L^{p,\infty}(\mu)$ denote the Lebesgue space and weak Lebesgue
space respectively. This result in this paper includes the
corresponding results on both the homogeneous spaces and
$(\mathbb R^{n},\mu)$ with non-doubling measures space. A variant of sharp
maximal operator $M^{\sharp}$, Kolmogorov's theorem and some good
properties of the dominating function $\lambda$ (see Definition 1.2)
are the main tools for proving the results of this paper.

Before stating the main results of this paper, we firstly recall some
notations and definitions.
\begin{defn}$^{[11]}$A metric space $(X,d)$
is called geometrically doubling if there exists some $N_{0}\in
\mathbf{N}$ such that, for any ball $B(x,r)\subset X$, there exists
a finite ball covering $\{B(x_i,r/2)\}_i$ of $B(x,r)$ such that the
cardinality of this covering is at most $N_0$.
 \end{defn}

\begin{defn}$^{[11]}$A metric measure space
$(X,d,\mu)$ is said to be upper doubling if $\mu$ is a Borel measure
on $X$ and there exists a dominating function $\lambda : X
\times(0,+\infty) \rightarrow (0,+\infty)$ and a constant
$C_{\lambda} >0$ such that for each $x\in X, r \longmapsto (x,r)$ is
non-decreasing, and for all $x\in X, r >0$,
 \begin{equation}
 \mu(B(x, r))\leq \lambda(x, r)\leq C_{\lambda}\lambda(x, r/2)
 \end{equation}
\end{defn}

\begin{rem}(i)\ A space of homogeneous type is a special case of upper
doubling spaces, where one can take the dominating function
$\lambda(x, r)= \mu(B(x,r))$. On the other hand, a metric space
$(X,d,\mu)$ satisfying the polynomial growth condition (1)(in
particular, $(X,d,\mu)=(\mathbb{R}^n, |\cdot|, \mu)$ with $\mu$
satisfying (1) for some $k\in (0, n])$) is also an upper doubling
measure space if we take $\lambda(x, r)=Cr^{k}$.

(ii)\ Let$(X,d,\mu)$ be an upper doubling space and $\lambda$ be a
dominating function on $X \times(0,+\infty)$ as in Definition 1.2. In
[14], it was showed that there exists another dominating function
$\tilde{\lambda}$ such that for all $x, y \in X$ with $d(x, y)\leq
r$,
\begin{equation}
\tilde{\lambda}(x, r)\leq \tilde{C}\tilde{\lambda}(y, r).
\end{equation}
 Thus, we suppose that $\lambda$ always satisfies (1.3) in this paper.
\end{rem}

\begin{defn}
 Let $\alpha,\beta \in (1,+\infty)$. A ball $B\subset X$ is called $(\alpha,
\beta)$-doubling if $\mu(\alpha B)\leq \beta \mu (B)$.
\end{defn}

As pointed in Lemma 2.3 of [1],  there exist plenties of doubling
balls with small radii and with large radii. In the rest of this
paper, unless $\alpha$ and $\beta$ are specified otherwise, by an
$(\alpha,\beta)$ doubling ball we mean a $(6,\beta_{0})$-doubling
with a fixed number $\beta_0 >\max\{C_{\lambda}^{3\log_{2}6},
6^{n}\}$, where $n=\log_{2}N_{0}$ is viewed as a geometric dimension
of the space.

\begin{defn}$^{[1]}$For any two balls
$B\subset Q$, define
\begin{equation}K_{B,Q} = 1+\int_{r_{B}\leq d(x,x_B)\leq r_Q}
\frac{d\mu(x)}{\lambda(x_B,d(x,x_B))}.
\end{equation}
And, for two balls $B\subset Q$, one define the coefficient
$K'_{B,Q}$ as follows. Let $N_{B,Q}$ be the smallest integer
satisfying $6^{N_{B,Q}}r_{B}\geq r_Q$, then we set
\begin{equation}
K'_{B,Q} = 1+\sum_{k=1}^{N_{B,Q}}
\frac{\mu(6^{k}B)}{\lambda(x_B,6^{k}r_{B})}.
\end{equation}
\end{defn}

\begin{rem} In the case that $\lambda(x, ar) =
a^{t}\lambda(x, r)$ for $0<t<\infty$, $ x \in X $,  and $ a,r > 0$, one know
that $K_{B,Q}\approx K'_{B,Q}$. However, in general, we only have
$K_{B,Q}\leq C K'_{B,Q}$. In this paper, we always suppose that
$\lambda(x, ar) = a^{t}\lambda(x, r)$ for $0<t<\infty$, $ x \in X $,  and $ a,r > 0$. So we don't differentiate $K_{B,Q}$ with $K'_{B,Q}$ and always
write $K_{B,Q}$ for simplicity in this paper.
\end{rem}

\begin{defn}
 A kernel $K(\cdot,\cdots,\cdot)\in L_{loc}^{1}((X)^{m+1}\backslash\{(x,y_{1}\cdots,y_{j},\cdots,y_{m}):x=y_{1}=\cdots=y_{j}=\cdots=y_{m}\})$
 is called an $m$-linear Calder\'{o}n-Zygmund kernel if it satisfies:

(i)\begin{equation}
 |K(x,y_{1},\cdots,y_{j},\cdots, y_{m})|\leq
C\biggl[\sum_{j=1}^{m}\lambda(x,d(x,y_{j}))\biggr]^{-m}
 \end{equation}
for all $(x,y_{1}\cdots,y_{j},\cdots,y_{m})\in (X)^{m+1}$ with
$x\neq y_{j}$ for some $j$.

 (ii) There exists $0<\delta\leq 1$ such that

\begin{equation}
\begin{split}
&|K(x,y_{1},\cdots,y_{j},\cdots,y_{m})-K(x',y_{1},\cdots,y_{j},\cdots,y_{m})|\\
\leq
&\frac{Cd(x,x')^{\delta}}{\biggl[\sum\limits_{j=1}^{m}d(x,y_{j})\biggr]^{\delta}\biggl[\sum\limits_{j=1}^{m}\lambda(x,d(x,y_{j}))\biggr]^{m}},
\end{split}
\end{equation}
provided that $Cd(x,x')\leq \max\limits_{1\leq j\leq m}d(x,y_{j})$ and
for each $j$,
 \begin{equation}
 \begin{split}
&|K(x,y_{1},\cdots,y_{j},\cdots,y_{m})-K(x,y_{1},\cdots,y'_{j},\cdots,y_{m})|\\
 &\leq\frac{Cd(y_{j},y'_{j})^{\delta}}{\biggl[\sum\limits_{j=1}^{m}d(x,y_{j})\biggr]^{\delta}
 \biggl[\sum\limits_{j=1}^{m}\lambda(x,d(x,y_{j}))\biggr]^{m}},
\end{split}
\end{equation}
provided that $Cd(y_{j},y'_{j})\leq \max\limits_{1\leq j\leq
m}d(x,y_{j})$.

A multilinear operator $T$ is called a multilinear
Calder\'{o}n-Zygmund singular integral operator with the above
kernel $K$ satisfying $(1.6)$, $(1.7)$ and $(1.8)$ if, for $f_{1},\cdots
f_{m}$ are $L^{\infty}$ functions with compact support and $x\notin
\bigcap_{j=1}^{m}\text{supp}f_{j}$,
 \begin{equation}
 T(f_{1},\cdots f_{m})(x)=\int_{X^{m}}K(x,y_{1},\cdots y_{m})f_{1}(y_{1})\cdots
 f_{m}(y_{m})d\mu(y_{1})\cdots d\mu(y_{m}).
\end{equation}
\end{defn}

\begin{rem}
Because $\max\limits_{1\leq j\leq m}d(x,y_{j})\leq
\sum\limits_{j=1}^{m}d(x,y_{j})\leq m\max\limits_{1\leq j\leq
m}d(x,y_{j})$, (ii) in Definition 1.7 is equivalent to (ii') in the
following statement.

(ii') There exists $0<\delta\leq 1$ such that

 \begin{equation}
 \begin{split}
 &|K(x,y_{1},\cdots,y_{j},\cdots,y_{m})-K(x',y_{1},\cdots,y_{j},\cdots,y_{m})|\\
 \leq & \frac{Cd(x,x')^{\delta}}{\biggl[\max\limits_{1\leq j\leq
m}d(x,y_{j})\biggr]^{\delta}\biggl[\sum\limits_{j=1}^{m}\lambda(x,d(x,y_{j}))\biggr]^{m}},
\end{split}
\end{equation}
provided that $Cd(x,x')\leq \max\limits_{1\leq j\leq m}d(x,y_{j})$ and
for each $j$,
\begin{equation}
 \begin{split}
&|K(x,y_{1},\cdots,y_{j},\cdots,y_{m})-K(x,y_{1},\cdots,y'_{j},\cdots,y_{m})|\\
 \leq &\frac{Cd(y_{j},y'_{j})^{\delta}}{\biggl[\max\limits_{1\leq j\leq
m}d(x,y_{j})\biggr]^{\delta}
 \biggl[\sum\limits_{j=1}^{m}\lambda(x,d(x,y_{j}))\biggr]^{m}},
\end{split}
\end{equation}
provided that $Cd(y_{j},y'_{j})\leq \max\limits_{1\leq j\leq
m}d(x,y_{j})$.
\end{rem}

\begin{defn}$^{[1]}$ Let $\rho>1$ be some
fixed constant. A function $b\in L_{loc}^{1}(\mu)$ is said to belong
to $RBMO(\mu)$ if there exists a constant $C
>0$ such that for any ball $B$
\begin{equation}
\frac{1}{\mu(\rho B)}\int_{B}|b(x)-m_{\widetilde{B}}b|d\mu(x)\leq C,
\end{equation}
and for any two doubling balls $B\subset Q$,
 \begin{equation}
  |m_{B}(b)-m_{Q}(b)|\leq CK_{B,Q},
\end{equation}
where $\widetilde{B}$ is the smallest $(\alpha,\beta)$-doubling ball of
the form $6^{k}B$ with $k\in {\mathbf{N}}\bigcup\{0\}$, and
$m_{\widetilde{B}}(b)$ is the mean value of $b$ on $\widetilde{B}$,
namely,
$$m_{\widetilde{B}}(b)=\frac{1}{\mu(\widetilde{B})}\int_{\widetilde{B}}b(x)d\mu(x).$$
The minimal constant $C$ appearing in (1.12) and (1.13) is defined to be
the $RBMO(\mu)$ norm of $b$ and denoted by $||b||_{\ast}$.
\end{defn}

For $1\leq i \leq k$, we denote by $C_{i}^{k}$ the family of all
finite subsets
$\sigma=\{\sigma(1),\sigma(2),\cdot\cdot\cdot,\sigma(i)\}$ of
$\{1,2,\cdot\cdot\cdot, k\}$ with $i$ different elements. For any
$\sigma\in C_{i}^{k}$, the complementary sequence $\sigma'$ is
given by $\sigma'=\{1,2,\cdot\cdot\cdot,k\}\backslash\sigma$.
Moreover, for $b_{i}\in RBMO(\mu),i=1,\cdots,k$, let
$\vec{b}=(b_{1},b_{2},\cdot\cdot\cdot,b_{k})$ be a finite family of
locally integrable function. For all $1\leq i\leq k$ and
$\sigma=\{\sigma(1),\cdot\cdot\cdot,\sigma(i)\}\in C_{i}^{k}$, we
set $\vec{b}_{\sigma}=(b_{\sigma(1)},\cdot\cdot\cdot,b_{\sigma(i)})$
and the product $b_{\sigma}(x)=b_{\sigma(1)}(x)\cdot\cdot\cdot
b_{\sigma(i)}(x)$. Also, we denote $\vec{f}=(f_{1},\cdots, f_{k})$,
$\vec{f}_{\sigma}=(f_{\sigma(1)},\cdots, f_{\sigma(i)})$ and
$\vec{b}_{\sigma'}\vec{f}_{\sigma'}=(b_{\sigma'(i+1)}f_{\sigma'(i+1)},\cdots,
b_{\sigma'(k)}f_{\sigma'(k)})$.

\begin{defn} A kind of commutators generated
by multilinear singular integral operator $T$ with $b_{i}\in
RBMO(\mu), i=1,\cdots, k$ is defined as follows:
\begin{equation}
[\vec{b},T](\vec{f})(x)=\sum_{i=0}^{k}\sum_{\sigma \in
C_{i}^{k}}(-1)^{k-i}b_{\sigma}(x)T(\vec{f}_{\sigma},\vec{b}_{\sigma'}\vec{f}_{\sigma'})(x).
\end{equation}
\end{defn}

In particular, when $k=2$, we can obtain
\begin{equation}
 \begin{split}
[b_{1},b_{2},T](f_{1},f_{2})(x)=&b_{1}(x)b_{2}(x)T(f_{1},f_{2})(x)-b_{1}(x)T(f_{1},b_{2}f_{2})(x)\\
&-b_{2}(x)T(b_{1}f_{1},f_{2})(x)+T(b_{1}f_{1},b_{2}f_{2})(x).
\end{split}
\end{equation}
Also, we define $[b_{1},T]$ and $[b_{2},T]$ as follows respectively.
\begin{equation}[b_{1},T](f_{1},f_{2})(x)=b_{1}(x)T(f_{1},f_{2})(x)-T(b_{1}f_{1},f_{2})(x),\end{equation}
\begin{equation}[b_{2},T](f_{1},f_{2})(x)=b_{2}(x)T(f_{1},f_{2})(x)-T(f_{1},b_{2}f_{2})(x).\end{equation}

For the sake of simplicity and without loss of generality, we only
consider the case of $k=2$ in this paper. Let us state the main
result as follows.

\begin{thm}
 Suppose that $\mu$ is a Radon measure with $||\mu||=\infty$. Let
$[b_{1},b_{2},T]$ defined by (1.15). Let $1<p_{1},p_{2}<+\infty$,
$f_{1}\in L^{p_{1}}(\mu)$, $f_{2}\in L^{p_{2}}(\mu)$, $b_{1}\in
RBMO(\mu)$ and $b_{2}\in RBMO(\mu)$. If $T$ is bounded from
$L^{1}(\mu)\times L^{1}(\mu)$ to $ L^{1/2,\infty}(\mu)$, then there
exists a constant $C>0$ such that
\begin{equation}
||[b_{1},b_{2},T](f_{1},f_{2})||_{L^{q}(\mu)}\leq C
||f_{1}||_{L^{p_{1}}(\mu)}||f_{2}||_{L^{p_{2}}(\mu)},
\end{equation}
where $\dfrac{1}{q}=\dfrac{1}{p_{1}}+\dfrac{1}{p_{2}}$.
\end{thm}

 Throughout this paper, $C$ always denotes a
positive constant independent of the main parameters involved, but
it may be different from line to line. And $p'$ is the conjugate
index of $p$, namely, $\dfrac{1}{p}+\dfrac{1}{p'}=1.$

\section{Proof of Main Result}

To prove the main theorem, we firstly give some notations and
lemmas.

Let $f \in L_{loc}^{1}(\mu)$, the sharp maximal operator is defined
by
 \begin{equation}
M^{\sharp}f(x)=\sup _{B\ni
x}\frac{1}{\mu(6B)}\int_{B}|f(y)-m_{\widetilde{B}}(f)|d\mu(y)
+\sup_{(B,Q)\in\Delta_{x}}\frac{|m_{B}(f)-m_{Q}(f)|}{K_{B,Q}},
\end{equation}
where $\Delta_{x}:=\{(B,Q):x\in B\subset Q\ \text{and}\ B,\ Q\ \text{are
doubling balls}\}$ and the non centered doubling maximal operator is
denoted by
$$Nf(x)=\sup_{B\ni x,\atop B \
\text{doubling}}\frac{1}{\mu(B)}\int_{B}|f(y)|d\mu(y).$$ For any $0<\delta
< 1$, we also define that
 \begin{equation}
M^{\sharp}_{\delta}f(x)=\{M^{\sharp}(|f|^{\delta})(x)\}^{1/\delta}
\end{equation}
and
\begin{equation}
N_{\delta}f(x)=\{N(|f|^{\delta})(x)\}^{1/\delta}.
\end{equation}
We can obtain that for any $f\in L^{1}_{loc}(\mu)$,
\begin{equation}
|f(x)|\leq N_{\delta}f(x)
\end{equation} for $\mu-a.e. \ x \in X$.
Let us give an explanation for inequality (2.4). By the Lebesgue
differential theorem, we obtain that $|f(x)|\leq Nf(x)$. Hence
$$|f(x)|=[|f(x)|^{\delta}]^{1/\delta}\leq \{N(|f|^{\delta})(x)\}^{1/\delta}=N_{\delta}f(x).$$

Let $\rho>1$, $p\in (1,\infty)$ and $r\in (1,p)$, the non-centered
maximal operator $M_{r,(\rho)}f$ is defined by
\begin{equation}
 M_{r,(\rho)}f(x)=\sup_{B\ni x}\biggl\{\frac{1}{\mu(\rho
B)}\int_{B}|f(y)|^{r}d\mu(y)\biggr\}^{1/r}.
\end{equation}
When $r=1$, we simply write $ M_{1,(\rho)}f(x)$ as $M_{(\rho)}f$. If
$\rho\geq 5$, then the operator $M_{(\rho)}f$ is bounded on
$L^{p}(\mu)$ for $p>1$ and $M_{r,(\rho)}$ is bounded on $L^{p}(\mu)$
for $p>r$ (see [1]).

From Theorem 4.2 in [1], it is easy to obtain that
\begin{lem}
Let $f \in L^{1}_{loc}(\mu)$ with $\int_{X}f(x)d\mu(x)=0$ if
$||\mu||<\infty$. For $1<p<\infty$ and $0<\delta<1$, if
$\inf(1,N_{\delta}f)\in L^{p}(\mu)$, then there exists a constant
$C>0$ such that
\begin{equation}
||N_{\delta}(f)||_{L^{p}(\mu)}\leq
C||M_{\delta}^{\sharp}(f)||_{L^{p}(\mu)}.
\end{equation}
\end{lem}

\begin{lem}$^{[4,19]}$
Let $1\leq p<\infty$ and $1<\rho <\infty$. Then $b\in RBMO(\mu)$ if
and only if for any ball $B\subset X$,
\begin{equation}
\biggl\{\frac{1}{\mu(\rho
B)}\int_{B}|b_{B}-m_{\tilde{B}}(b)|^{p}d\mu(x)\biggr\}^{1/p}\leq
C||b||_{\ast},\end{equation} and for any two doubling balls
$B\subset Q$,
\begin{equation}
|m_{B}(b)-m_{Q}(b)|\leq C K_{B,Q}||b||_{\ast}.
\end{equation}
\end{lem}

\begin{lem}$^{[4]}$
\begin{equation}
|m_{\widetilde{6^{j}\frac{6}{5}B}}(b)-m_{\tilde{B}}(b)|\leq
Cj||b||_{\ast}.
\end{equation}
\end{lem}

\begin{lem}$^{[10]}$
Suppose that $\mu$ is a Radon measure with $||\mu||=\infty$. Let
$T$ be defined by (1.9) with $m=2$. Let $1<p_{1},p_{2}<+\infty$, $f_{1}\in
L^{p_{1}}(\mu)$ and $f_{2}\in L^{p_{2}}(\mu)$. If $T$ is bounded
from $L^{1}(\mu)\times L^{1}(\mu)$ to $ L^{1/2,\infty}(\mu)$, then
there exists a constant $C>0$ such that
\begin{equation}
||T(f_{1},f_{2})||_{L^{q}(\mu)}\leq C
||f_{1}||_{L^{p_{1}}(\mu)}||f_{2}||_{L^{p_{2}}(\mu)},
\end{equation}
where $\dfrac{1}{q}=\dfrac{1}{p_{1}}+\dfrac{1}{p_{2}}$.
\end{lem}

\begin{lem}
Suppose that $[b_{1},b_{2},T]$ is defined by (1.15), $0<\delta<1/2$,
$1<p_{1},p_{2},q<\infty$, $1<r<q$ and $b_{1},b_{2}\in RBMO(\mu)$. If
$T$ is bounded from $L^{1}(\mu)\times L^{1}(\mu)$ to $
L^{1/2,\infty}(\mu)$, then there exists a constant $C>0$ such that
for any $x\in X$, $f_{1}\in L^{p_{1}}(\mu)$ and $f_{2}\in
L^{p_{2}}(\mu)$,
\begin{equation}
\begin{split}
&M_{\delta}^{\sharp}[b_{1},b_{2},T](f_{1},f_{2})(x)\leq
C||b_{1}||_{\ast}||b_{2}||_{\ast}M_{r,(6)}(T(f_{1},f_{2}))(x)\\
&\ \ +C||b_{1}||_{\ast}M_{r,(6)}([b_{2},T](f_{1},f_{2}))(x)
+C||b_{2}||_{\ast}M_{r,(6)}([b_{1},T](f_{1},f_{2}))(x)\\
&\ \
+C||b_{1}||_{\ast}||b_{2}||_{\ast}M_{p_{1},(5)}f_{1}(x)M_{p_{2},(5)}f_{2}(x),
\end{split}
\end{equation}

\begin{equation}
\begin{split}
M_{\delta}^{\sharp}[b_{1},T](f_{1},f_{2})(x)\leq
C||b_{1}||_{\ast}M_{r,(6)}(T(f_{1},f_{2}))(x)
+C||b_{1}||_{\ast}M_{p_{1},(5)}f_{1}(x)M_{p_{2},(5)}f_{2}(x),
\end{split}
\end{equation}
and
\begin{equation}
\begin{split}
M_{\delta}^{\sharp}[b_{2},T](f_{1},f_{2})(x)\leq
C||b_{2}||_{\ast}M_{r,(6)}(T(f_{1},f_{2}))(x)
+C||b_{2}||_{\ast}M_{p_{1},(5)}f_{1}(x)M_{p_{2},(5)}f_{2}(x).
\end{split}
\end{equation}
\end{lem}

\begin{proof}
Because $L^{\infty}(\mu)$ with compact support is dense in
$L^{p}(\mu)$ for $1<p<\infty$, we only consider the situation of
$f_{1},f_{2}\in L^{\infty}(\mu)$ with compact support. Also, by
Corollary 3.11 in [4], without loss of generality, we can assume
that $b_{1},b_{2}\in L^{\infty}(\mu)$.

 As in the proof of Theorem 9.1 in [19], to obtain (2.11), it suffices to show that
\begin{equation}
\begin{split}
&\biggl(\frac{1}{\mu(6B)}\int_{B}||[b_{1},b_{2},T](f_{1},f_{2})(z)|^{\delta}-|h_{B}|^{\delta}|d\mu(z)\biggr
)^{1/\delta}\\
\leq &C||b_{1}||_{\ast}||b_{2}||_{\ast}M_{r,(6)}(T(f_{1},f_{2}))(x)
+C||b_{1}||_{\ast}M_{r,(6)}([b_{2},T](f_{1},f_{2}))(x)\\
&+C||b_{2}||_{\ast}M_{r,(6)}([b_{1},T](f_{1},f_{2}))(x)+C||b_{1}||_{\ast}||b_{2}||_{\ast}M_{p_{1},(5)}f_{1}(x)M_{p_{2},(5)}f_{2}(x),
\end{split}
\end{equation}
 holds for
any $x$ and ball $B$ with $x\in B$, and
\begin{equation}
\begin{split}
|h_{B}-h_{Q}|& \leq
CK_{B,Q}^{2}\biggr[||b_{1}||_{\ast}||b_{2}||_{\ast}M_{r,(6)}(T(f_{1},f_{2}))(x)
+||b_{1}||_{\ast}M_{r,(6)}([b_{2},T](f_{1},f_{2}))(x)\\
&\ \ +||b_{2}||_{\ast}M_{r,(6)}([b_{1},T](f_{1},f_{2}))(x)
+||b_{1}||_{\ast}||b_{2}||_{\ast}M_{p_{1},(5)}f_{1}(x)M_{p_{2},(5)}f_{2}(x)\biggr].
\end{split}
\end{equation}
for all balls $B\subset Q$ with $x\in B$, where $B$ is an arbitrary
ball, $Q$ is a doubling ball. For any ball $B$, we denote
$$h_{B}:= m_{B}(T((b_{1}-
m_{\tilde{B}}(b_{1}))f_{1}\chi_{X\backslash\frac{6}{5}B},(b_{2}-m_{\tilde{B}}(b_{2}))f_{2}\chi_{X\backslash\frac{6}{5}B})),$$
and
$$h_{Q}:= m_{Q}(T((b_{1}-
m_{Q}(b_{1}))f_{1}\chi_{X\backslash\frac{6}{5}Q},(b_{2}-m_{Q}(b_{2}))f_{2}\chi_{X\backslash\frac{6}{5}Q})).$$

Write
$$[b_{1},b_{2},T]=T((b_{1}-b_{1}(z))f_{1},(b_{2}-b_{2}(z))f_{2}),$$
and
\begin{equation}
\begin{split}
&T((b_{1}-m_{\tilde{B}}(b_{1}))f_{1},(b_{2}-m_{\tilde{B}}(b_{2}))f_{2})\\
=&T((b_{1}-b_{1}(z)+b_{1}(z)-m_{\tilde{B}}(b_{1}))f_{1},(b_{2}-b_{2}(z)+b_{2}(z)-m_{\tilde{B}}(b_{2}))f_{2})\\
=&(b_{1}(z)-m_{\tilde{B}}(b_{1}))(b_{2}(z)-m_{\tilde{B}}(b_{2}))T(f_{1},f_{2})\\
&\   -(b_{1}(z)-m_{\tilde{B}}(b_{1}))T(f_{1},(b_{2}-b_{2}(z))f_{2})\\
&\
-(b_{2}(z)-m_{\tilde{B}}(b_{2}))T((b_{1}-b_{1}(z))f_{1},f_{2})+T((b_{1}-b_{1}(z))f_{1},(b_{2}-b_{2}(z))f_{2}).
\end{split}
\end{equation}
Then
\begin{equation}
\begin{split}
&\biggl(\frac{1}{\mu(6B)}\int_{B}||[b_{1},b_{2},T](f_{1},f_{2})(z)|^{\delta}-|h_{B}|^{\delta}d\mu(z)\biggr
)^{1/\delta}\\
\leq&
C\biggl(\frac{1}{\mu(6B)}\int_{B}|[b_{1},b_{2},T](f_{1},f_{2})(z)-h_{B}|^{\delta}d\mu(z)\biggr)^{1/\delta}\\
\leq&
C\biggl(\frac{1}{\mu(6B)}\int_{B}|(b_{1}(z)-m_{\tilde{B}}(b_{1}))(b_{2}(z)-m_{\tilde{B}}(b_{2}))T(f_{1},f_{2})(z)|^{\delta}d\mu(z)\biggr)^{1/\delta}\\
&\ \ +C\biggl(\frac{1}{\mu(6B)}\int_{B}|(b_{1}(z)-m_{\tilde{B}}(b_{1}))T(f_{1},(b_{2}-b_{2}(z))f_{2})(z)|^{\delta}d\mu(z)\biggr)^{1/\delta}\\
&\ \ +C\biggl(\frac{1}{\mu(6B)}\int_{B}|(b_{2}(z)-m_{\tilde{B}}(b_{2}))T((b_{1}-b_{1}(z))f_{1},f_{2})(z)|^{\delta}d\mu(z)\biggr)^{1/\delta}\\
&\ \
+C\biggl(\frac{1}{\mu(6B)}\int_{B}|T((b_{1}-m_{\tilde{B}}(b_{1}))f_{1},(b_{2}-m_{\tilde{B}}(b_{2}))f_{2})(z)-h_{B}|^{\delta}d\mu(z)\biggr)^{1/\delta}\\
=&:E_{1}+E_{2}+E_{3}+E_{4}.
\end{split}
\end{equation}
We firstly estimate $E_{1}$. Let $r_{1},r_{2}>1$ such that
$\dfrac{1}{r}+\dfrac{1}{r_{1}}+\dfrac{1}{r_{2}}=\dfrac{1}{\delta}$.
By H\"{o}lder's inequality and Lemma 2.2, we yields

\begin{equation}
\begin{split}
E_{1}&\leq
C\biggl(\frac{1}{\mu(6B)}\int_{B}|b_{1}(z)-m_{\tilde{B}}b_{1}|^{r_{1}}d\mu(z)\biggr)^{1/r_{1}}\\
&\ \ \times\biggl(\frac{1}{\mu(6B)}\int_{B}|b_{2}(z)-m_{\tilde{B}}b_{2}|^{r_{2}}d\mu(z)\biggr)^{1/r_{2}}\\
&\ \ \times\biggl(\frac{1}{\mu(6B)}\int_{B}|T(f_{1},f_{2})|^{r}d\mu(z)\biggr)^{1/r}\\
&\leq C||b_{1}||_{\ast}||b_{2}||_{\ast}M_{r,(6)}(T(f_{1},f_{2}))(x).
\end{split}
\end{equation}

For $E_{2}$, let $s>1$ such that
$\dfrac{1}{s}+\dfrac{1}{r}=\dfrac{1}{\delta}$, by H\"{o}lder's
inequality and Lemma 2.2, we deduce

\begin{equation}
\begin{split}
E_{2}&\leq
C\biggl(\frac{1}{\mu(6B)}\int_{B}|b_{1}(z)-m_{\tilde{B}}b_{1}|^{s}d\mu(z)\biggr)^{1/s}\\
&\ \ \times\biggl(\frac{1}{\mu(6B)}\int_{B}|[b_{2},T](f_{1},f_{2})|^{r}d\mu(z)\biggr)^{1/r}\\
&\leq C||b_{1}||_{\ast}M_{r,(6)}([b_{2},T](f_{1},f_{2}))(x).
\end{split}
\end{equation}

Similar to estimate $E_{2}$, we immediately get
\begin{equation}
E_{3}\leq C||b_{2}||_{\ast}M_{r,(6)}([b_{1},T](f_{1},f_{2}))(x).
\end{equation}

Let us turn to estimate $E_{4}$. Denote
$f_{j}^{1}=f_{j}\chi_{\frac{6}{5}B}$ and $f_{j}^{2}=f_{j}-f_{j}^{1}$
for $j=1,2$. Then
\begin{equation}
\begin{split}
E_{4}&\leq
C\biggl(\frac{1}{\mu(6B)}\int_{B}|T((b_{1}-m_{\tilde{B}}b_{1})f_{1}^{1}(z),(b_{2}-m_{\tilde{B}}b_{2})f_{2}^{1})(z)|^{\delta}d\mu(z)\biggr)^{1/\delta}\\
&\ \ \ +
C\biggl(\frac{1}{\mu(6B)}\int_{B}|T((b_{1}-m_{\tilde{B}}b_{1})f_{1}^{1}(z),(b_{2}-m_{\tilde{B}}b_{2})f_{2}^{2})(z)|^{\delta}d\mu(z)\biggr)^{1/\delta}\\
&\ \ \ +
C\biggl(\frac{1}{\mu(6B)}\int_{B}|T((b_{1}-m_{\tilde{B}}b_{1})f_{1}^{2}(z),(b_{2}-m_{\tilde{B}}b_{2})f_{2}^{1})(z)|^{\delta}d\mu(z)\biggr)^{1/\delta}\\
&\ \ \ +
C\biggl(\frac{1}{\mu(6B)}\int_{B}|T((b_{1}-m_{\tilde{B}}b_{1})f_{1}^{2}(z),(b_{2}-m_{\tilde{B}}b_{2})f_{2}^{2})(z)-h_{B}|^{\delta}d\mu(z)\biggr)^{1/\delta}\\
&=:E_{41}+E_{42}+E_{43}+E_{44}.
\end{split}
\end{equation}

To estimate $E_{41}$, we need the classical Kolmogorov's theorem:
Let $(X,\mu)$ be a probability measure space and let $0<p<q<\infty$,
then there exists a constant $C>0$, such that $||f||_{L^{p}(\mu)}\le
C||f||_{L^{q,\infty}(\mu)}$ for any measurable function $f$. Let
$p=\delta$ and $q=1/2$ such that $0<\delta<1/2$. Using Kolmogorov's
theorem, Lemma 2.2 and H\"{o}lder's inequality, we obtain
\begin{equation}
\begin{split}
&E_{41}\leq
C||T((b_{1}-m_{\tilde{B}}b_{1})f_{1}^{1},(b_{2}-m_{\tilde{B}}b_{2})f_{2}^{1})||_{L^{1/2,\infty}(\frac{6}{5}B,\frac{d\mu(z)}{\mu(6B)})}\\
\leq&
C\frac{1}{\mu(6B)}\int_{\frac{6}{5}B}|(b_{1}-m_{\tilde{B}}b_{1})f_{1}(z)|d\mu(z)\\
&\ \ \times\frac{1}{\mu(6B)}\int_{\frac{6}{5}B}|(b_{2}-m_{\tilde{B}}b_{2})f_{2}(z)|d\mu(z)\\
\leq&
C(\frac{1}{\mu(6B)}\int_{\frac{6}{5}B}|(b_{1}-m_{\tilde{B}}b_{1}|^{p'_{1}}d\mu(z))^{1/p'_{1}}\\
&\ \ \times(\frac{1}{\mu(6B)}\int_{\frac{6}{5}B}|f_{1}(z)|^{p_{1}}d\mu(z))^{1/p_{1}}\\
&\ \ \times(\frac{1}{\mu(6B)}\int_{\frac{6}{5}B}|(b_{2}-m_{\tilde{Q}}b_{2})|^{p'_{2}}d\mu(z))^{1/p'_{2}}\\
&\ \ \times(\frac{1}{\mu(6B)}\int_{\frac{6}{5}B}|f_{2}(z)|^{p_{2}}d\mu(z))^{1/p_{2}}\\
\leq&
C||b_{1}||_{\ast}||b_{2}||_{\ast}M_{p_{1},(5)}f_{1}(x)M_{p_{2},(5)}f_{2}(x).
\end{split}
\end{equation}

To compute $E_{42}$, using (i) of Definition 1.7, Lemma 2.1, Lemma 2.2,
H\"{o}lder's inequality and the properties of $\lambda$, we know

\begin{equation}
\begin{split}
E_{42}&\leq
C\frac{1}{\mu(6B)}\int_{X}\int_{X}\int_{X\backslash\frac{6}{5}B}
\frac{|b_{1}(y_{1})-m_{\tilde{B}}b_{1}||f_{1}^{1}(y_{1})|}
{[\lambda(z,d(z,y_{1}))+\lambda(z,d(z,y_{2}))]^{2}}\\
&\ \ \ \ \times|b_{2}(y_{2})-m_{\tilde{B}}b_{2}||f_{1}^{2}(y_{2})|d\mu(y_{1})d\mu(y_{2})d\mu(z)\\
&\leq C\frac{1}{\mu(6B)}\int_{B}\int_{\frac{6}{5}B}
|b_{1}(y_{1})-m_{\tilde{B}}b_{1}||f_{1}(y_{1})|d\mu(y_{1})\\
&\ \ \ \times\int_{X\backslash\frac{6}{5}B}
\frac{|b_{2}(y_{2})-m_{\tilde{B}}b_{2}||f_{2}(y_{2})|d\mu(y_{2})}{[\lambda(z,d(z,y_{2}))]^{2}}d\mu(z)\\
&\leq
C(\frac{1}{\mu(6B)}\int_{\frac{6}{5}B}|b_{1}(y_{1})-m_{\tilde{B}}b_{1}|^{p'_{1}}d\mu(y_{1}))^{1/p'_{1}}\\
&\ \ \ \times(\frac{1}{\mu(6B)}\int_{\frac{6}{5}B}|f_{1}(y_{1})|^{p_{1}}d\mu(y_{1}))^{1/p_{1}}\\
&\ \ \ \times
\mu(B)\sum_{k=1}^{\infty}\int_{6^{k}\frac{6}{5}B}\frac{|b_{2}(y_{2})-m_{\tilde{B}}
b_{2}||f_{2}(y_{2})|}{[\lambda(z,6^{k-1}\frac{6}{5}r_{B})]^{2}}d\mu(y_{2})\\
&\leq
C||b_{1}||_{\ast}M_{p_{1},(5)}f_{1}(x)\sum_{k=1}^{\infty}6^{-km}\frac{\mu(B)}{\mu(\frac{6}{5}B)}\frac{\mu(\frac{6}{5}B)}{\lambda(z,\frac{6}{5}r_{B})}\\
&\ \ \ \times\frac{1}{\lambda(z,6^{k-1}\frac{6}{5}r_{B})}
\int_{6^{k}\frac{6}{5}B}|b_{2}(y_{2})-m_{\tilde{B}}
b_{2}||f_{2}(y_{2})|d\mu(y_{2})\\
&\leq
C||b_{1}||_{\ast}M_{p_{1},(5)}f_{1}(x)\sum_{k=1}^{\infty}6^{-km}\frac{1}{\mu(5\times6^{k}\frac{6}{5}B)}\\
&\ \ \ \times\int_{6^{k}\frac{6}{5}B}
|b_{2}(y_{2})-m_{\widetilde{6^{k}\frac{6}{5}B}}(b_{2})+m_{\widetilde{6^{k}\frac{6}{5}B}}(b_{2})-m_{\tilde{B}}b_{2}||f_{2}(y_{2})|d\mu(y_{2})\\
&\leq
C||b_{1}||_{\ast}M_{p_{1},(5)}f_{1}(x)\sum_{k=1}^{\infty}6^{-km}\biggl[\biggl(\frac{1}{\mu(5\times6^{k}\frac{6}{5}B)}\\
&\ \ \ \times\int_{6^{k+1}\frac{6}{5}B}|b_{2}(y_{2})-m_{\widetilde{6^{k}\frac{6}{5}B}}(b_{2})|^{p'_{2}}d\mu(y_{2})\biggr)^{1/p'_{2}}\\
&\ \ \ \times
\biggl(\frac{1}{\mu(5\times6^{k}\frac{6}{5}B)}\int_{6^{k}\frac{6}{5}B}|f_{2}(y_{2})|^{p_{2}}d\mu(y_{2})\biggr)^{1/p_{2}}\\
&\ \ \ +C
k||b_{2}||_{\ast}\frac{1}{\mu(5\times6^{k}\frac{6}{5}B)}\int_{6^{k}\frac{6}{5}B}|f_{2}(y_{2})|d\mu(y_{2})\biggr]\\
&\leq
C||b_{1}||_{\ast}||b_{2}||_{\ast}M_{p_{1},(5)}f_{1}(x)M_{p_{2},(5)}f_{2}(x).\\
\end{split}
\end{equation}
Similarly, we get
\begin{equation}
E_{43}\leq
C||b_{1}||_{\ast}||b_{2}||_{\ast}M_{p_{1},(5)}f_{1}(x)M_{p_{2},(5)}f_{2}(x).
\end{equation}

For $E_{44}$, by (ii) of Definition 1.7, Lemma 2.1, Lemma 2.2,
H\"{o}lder's inequality and the properties of $\lambda$, we obtain
\begin{equation}
\begin{split}
&|T((b_{1}-m_{\tilde{B}}b_{1})f_{2}^{2},(b_{2}-m_{\tilde{B}}b_{2})f_{2}^{2})(z)
-T((b_{1}-m_{\tilde{B}}b_{1})f_{2}^{2},(b_{2}-m_{\tilde{B}}b_{2})f_{2}^{2})(z_{0})|\\
\leq &C\int_{X\backslash
\frac{6}{5}B}\int_{X\backslash\frac{6}{5}B}|K(z,y_{1},y_{2})-K(z_{0},y_{1},y_{2})|\\
&\ \ \ \times|\prod_{i=1}^{2}(b_{i}(y_{i})-m_{\tilde{B}}b_{i})f_{i}(y_{i})|d\mu(y_{1})d\mu(y_{2})\\
\leq &C\int_{X\backslash
\frac{6}{5}B}\int_{X\backslash\frac{6}{5}B}\frac{d(z,z_{0})^{\delta}
|\prod_{i=1}^{2}(b_{i}(y_{i})-m_{\tilde{B}}b_{i})f_{i}(y_{i})|d\mu(y_{1})d\mu(y_{2})}
{(d(z,y_{1})+d(z,y_{2}))^{\delta}[\sum_{j=1}^{2}\lambda(z,d(z,y_{j}))]^{2}}\\
\leq &C\prod_{i=1}^{2}\int_{X\backslash
\frac{6}{5}B}\frac{d(z,z_{0})^{\delta_{i}}
|b_{i}(y_{i})-m_{\tilde{B}}b_{i}||f_{i}(y_{i})d\mu(y_{i})|}{d(z,y_{i})^{\delta_{i}}\lambda(z,d(z,y_{i}))}\\
\leq
&C\prod_{i=1}^{2}\sum_{k=1}^{\infty}\int_{6^{k}\frac{6}{5}B}6^{-k\delta_{i}}\frac{\mu(5\times6^{k}\frac{6}{5}B)}
{\lambda(z,5\times6^{k}\frac{6}{5}r_{B})}\frac{1}{\mu(5\times6^{k}\frac{6}{5}B)}|b_{i}(y_{i})-m_{\tilde{B}}b_{i}||f_{i}|d\mu(y_{i})\\
\leq
&C\prod_{i=1}^{2}\sum_{k=1}^{\infty}6^{-k\delta_{i}}(\frac{1}{\mu(5\times6^{k}\frac{6}{5}B)}
\int_{6^{k}\frac{6}{5}B}|b_{i}(y_{i})-m_{\tilde{B}}b_{i}|^{p'_{i}}d\mu(y_{i}))^{1/p'_{i}}\\
&\ \ \ \times(\frac{1}{\mu(5\times6^{k}\frac{6}{5}B)}
\int_{6^{k}\frac{6}{5}B}|f_{i}|^{p_{i}})^{1/p_{i}}\\
\leq
&C\prod_{i=1}^{2}\sum_{k=1}^{\infty}6^{-k\delta_{i}}M_{p_{i},(6)}f_{i}(x)(\frac{1}{\mu(5\times6^{k}\frac{6}{5}B)}
\int_{6^{k}\frac{6}{5}B}|b_{i}(y_{i})-m_{\widetilde{6^{k}\frac{6}{5}B}}\\
&\ \ \ +m_{\widetilde{6^{k}\frac{6}{5}B}}-m_{\tilde{B}}b_{i}|^{p'_{i}}d\mu(y_{i}))^{1/p'_{i}}\\
\leq
&C\prod_{i=1}^{2}\sum_{k=1}^{\infty}6^{-k\delta_{i}}k||b_{i}||_{\ast}M_{p_{i},(5)}f_{i}(x)\\
\leq&
C||b_{1}||_{\ast}||b_{2}||_{\ast}M_{p_{1},(5)}f_{1}(x)M_{p_{2},(5)}f_{2}(x).
\end{split}
\end{equation}
where $\delta_{1},\delta_{2}>0$ and $\delta_{1}+\delta_{2}=\delta$.

Taking the mean over $z_{0}\in B$, we deduce
\begin{equation}
E_{44}\leq
C||b_{1}||_{\ast}||b_{2}||_{\ast}M_{p_{1},(5)}f_{1}(x)M_{p_{2},(5)}f_{2}(x).
\end{equation}
So (2.14) can be obtain from (2.17) to (2.26).

Next we prove (2.15). Consider two balls $B\subset Q$ with $x\in B$,
where $B$ is an arbitrary ball and $Q$ is a doubling ball. Let
$N=N_{B,Q}+1$, then we obtain

\begin{equation}
\begin{split}
&\biggl||m_{B}T((b_{1}-m_{\tilde{B}}b_{1})f_{1}^{2},(b_{2}-m_{\tilde{B}}b_{2})f_{2}^{2})|\\
&\ \ -|m_{Q}T((b_{1}-m_{Q}b_{1})f_{1}^{2},(b_{2}-m_{Q}b_{2})f_{2}^{2})|\biggr|\\
\leq
&|m_{B}T((b_{1}-m_{\tilde{B}}b_{1})f_{1}\chi_{X\backslash6^{N}B},(b_{2}-m_{\tilde{B}}b_{2})f_{2}\chi_{X\backslash6^{N}B})\\
&\ \ -m_{Q}T((b_{1}-m_{\tilde{B}}b_{1})f_{1}\chi_{X\backslash6^{N}B},(b_{2}-m_{\tilde{B}}b_{2})f_{2}\chi_{X\backslash6^{N}B})|\\
&\ \
+|m_{Q}T((b_{1}-m_{Q}b_{1})f_{1}\chi_{X\backslash6^{N}B},(b_{2}-m_{Q}b_{2})f_{2}\chi_{X\backslash6^{N}B})\\
&\ \
-m_{Q}T((b_{1}-m_{\tilde{B}}b_{1})f_{1}\chi_{X\backslash6^{N}B},(b_{2}-m_{\tilde{B}}b_{2})f_{2}\chi_{X\backslash6^{N}B})|\\
&\ \
+|m_{B}T((b_{1}-m_{\tilde{B}}b_{1})f_{1}\chi_{6^{N}B\backslash\frac{6}{5}B},(b_{2}-m_{\tilde{B}}b_{2})f_{2}\chi_{X\backslash\frac{6}{5}B})\\
&\ \
+|m_{B}T((b_{1}-m_{\tilde{B}}b_{1})f_{1}\chi_{X\backslash\frac{6}{5}B},(b_{2}-m_{\tilde{B}}b_{2})f_{2}\chi_{6^{N}B\backslash\frac{6}{5}B})\\
&\ \
+|m_{Q}T((b_{1}-m_{Q}b_{1})f_{1}\chi_{6^{N}B\backslash\frac{6}{5}Q},(b_{2}-m_{Q}b_{2})f_{2}\chi_{X\backslash6^{N}B})\\
&\ \
+|m_{Q}T((b_{1}-m_{Q}b_{1})f_{1}\chi_{X\backslash\frac{6}{5}Q},(b_{2}-m_{Q}b_{2})f_{2}\chi_{6^{N}B\backslash\frac{6}{5}Q})\\
=:&F_{1}+F_{2}+F_{3}+F_{4}+F_{5}+F_{6}.
\end{split}
\end{equation}

Using the method to estimate $E_{4}$, we get
\begin{equation}
F_{1}\leq
CK_{B,Q}^{2}||b_{1}||_{\ast}||b_{2}||_{\ast}M_{p_{1},(5)}f_{1}(x)M_{p_{2},(5)}f_{2}(x).
\end{equation}

Let us estimate $F_{2}$. At first, we compute

\begin{equation}
\begin{split}
&T((b_{1}-m_{Q}b_{1})f_{1}\chi_{X\backslash6^{N}B},(b_{2}-m_{Q}b_{2})f_{2}\chi_{X\backslash6^{N}B})(z)\\
&\ \ -T((b_{1}-m_{\tilde{B}}b_{1})f_{1}\chi_{X\backslash6^{N}B},(b_{2}-m_{\tilde{B}}b_{2})f_{2}\chi_{X\backslash6^{N}B})(z)\\
=&(m_{Q}b_{2}-m_{\tilde{B}}b_{2})T((b_{1}-m_{Q}b_{1})f_{1}\chi_{X\backslash6^{N}B},f_{2}\chi_{X\backslash6^{N}B})(z)\\
&\ \ +(m_{Q}b_{1}-m_{\tilde{B}}b_{1})T(f_{1}\chi_{X\backslash6^{N}B},(b_{2}-m_{Q}b_{2})f_{2}\chi_{X\backslash6^{N}B})(z)\\
&\ \ +(m_{Q}b_{1}-m_{\tilde{B}}b_{1})(m_{Q}b_{2}-m_{\tilde{B}}b_{2})T(f_{1}\chi_{X\backslash6^{N}B},f_{2}\chi_{X\backslash6^{N}B})(z).\\
\end{split}
\end{equation}
Hence
\begin{equation}
\begin{split}
F_{2}\leq&|(m_{Q}b_{2}-m_{\tilde{B}}b_{2})\frac{1}{\mu(Q)}\int_{Q}T((b_{1}-m_{Q}b_{1})f_{1}
\chi_{X\backslash6^{N}B},f_{2}\chi_{X\backslash6^{N}B})(z)d\mu(z)|\\
&\ +|(m_{Q}b_{1}-m_{\tilde{B}}b_{1})\frac{1}{\mu(Q)}\int_{Q}T((f_{1}\chi_{X\backslash6^{N}B},(b_{2}-m_{Q}b_{2})f_{2}\chi_{X\backslash6^{N}B})(z)d\mu(z)|\\
&\
+|(m_{Q}b_{1}-m_{\tilde{B}}b_{1})(m_{Q}b_{2}-m_{\tilde{B}}b_{2})\frac{1}{\mu(Q)}\int_{Q}
T(f_{1}\chi_{X\backslash6^{N}B},f_{2}\chi_{X\backslash6^{N}B})(z)|\\
=:&F_{21}+F_{22}+F_{23}.
\end{split}
\end{equation}
To estimate $F_{21}$, we write
\begin{equation}
\begin{split}
&T((b_{1}-m_{Q}b_{1})f_{1}\chi_{X\backslash6^{N}Q},f_{2}\chi_{X\backslash6^{N}Q})(z)\\
=&T((b_{1}-m_{Q}b_{1})f_{1},f_{2})(z)-T((b_{1}-m_{Q}b_{1})f_{1}\chi_{6^{N}B}\chi_{\frac{6}{5}Q},f_{2}\chi_{\frac{6}{5}Q})(z))\\
&\ -T((b_{1}-m_{Q}b_{1})f_{1}\chi_{\frac{6}{5}Q},f_{2}\chi_{6^{N}B}\chi_{\frac{6}{5}Q})(z))\\
&\ +T((b_{1}-m_{Q}b_{1})f_{1}\chi_{6^{N}B}\chi_{\frac{6}{5}Q},f_{2}\chi_{6^{N}B}\chi_{\frac{6}{5}Q})(z)\\
&\ -T((b_{1}-m_{Q}b_{1})f_{1}\chi_{X\backslash\frac{6}{5}Q},f_{2}\chi_{6^{N}B})(z)\\
&\ -T((b_{1}-m_{Q}b_{1})f_{1}\chi_{6^{N}B},f_{2}\chi_{X\backslash\frac{6}{5}Q})(z)\\
&\
 +T((b_{1}-m_{Q}b_{1})f_{1}\chi_{6^{N}B\backslash\frac{6}{5}Q},f_{2}\chi_{6^{N}B\backslash\frac{6}{5}Q})(z)\\
 =:&H_{1}(z)+H_{2}(z)+H_{3}(z)+H_{4}(z)+H_{5}(z)+H_{6}(z)+H_{7}(z).
\end{split}
\end{equation}
Let us estimate $H_{1}(z)$ firstly. Since
$$\frac{1}{\mu(Q)}\int_{Q}|T(b_{1}-b_{1}(z)f_{1},f_{2})(z)|d\mu(z)\leq CM_{r,(6)}([b_{1},T]f_{1},f_{2})(x)$$
and by H\"{o}lder's inequality, we have
$$\frac{1}{\mu(Q)}\int_{Q}|(b_{1}(z)-m_{Q}(b_{1}))T(f_{1},f_{2})(z)|d\mu(z)\leq C||b_{1}||_{\ast}M_{r,(6)}(T(f_{1},f_{2}))(x),$$
then we obtain
\begin{equation}
\begin{split}
|m_{Q}(H_{1})|&\leq
|m_{Q}(T(b_{1}-b_{1}(z)f_{1},f_{2}))|+|m_{Q}((b_{1}(z)-m_{Q}(b_{1}))T(f_{1},f_{2}))|\\
&\leq
CM_{r,(6)}([b_{1},T]f_{1},f_{2})(x)+||b_{1}||_{\ast}M_{r,(6)}(T(f_{1},f_{2}))(x).
\end{split}
\end{equation}
For $H_{2}(z)$, let $r>1$ and $1<s_{1}<p_{1}$ such that
$\dfrac{1}{r}=\dfrac{1}{s_{1}}+\dfrac{1}{p_{2}}$. Denote
$\dfrac{1}{s_{2}}=\dfrac{1}{s_{1}}-\dfrac{1}{p_{1}}$, using the fact
of $Q$ is a doubling balls, Kolmogorov's inequality, H\"{o}lder's
inequality and Lemma 2.4, we yield
\begin{equation}
\begin{split}
|m_{Q}(H_{2})|\leq& C||H_{2}||_{L^{r,\infty}(Q,\frac{d\mu(z)}{\mu(Q)})}\\
\leq&
C(\frac{1}{\mu(Q)}\int_{\frac{6}{5}Q}|b_{1}-m_{Q}b_{1}|^{s_{1}}d\mu(z))^{1/s_{1}}
(\frac{1}{\mu(Q)}\int_{\frac{6}{5}Q}|f_{2}|^{p_{2}}d\mu(z))^{1/p_{2}}\\
\leq&
C(\frac{1}{\mu(6Q)}\int_{\frac{6}{5}Q}|b_{1}-m_{Q}b_{1}|^{s_{2}}d\mu(z))^{1/s_{2}}
(\frac{1}{\mu(6Q)}\int_{\frac{6}{5}Q}|f_{2}|^{p_{1}}d\mu(z))^{1/p_{1}}\\
&\ \ \times(\frac{1}{\mu(6Q)}\int_{\frac{6}{5}Q}|f_{2}|^{p_{2}}d\mu(z))^{1/p_{2}}\\
\leq& ||b_{1}||_{\ast}M_{p_{1},(5)}f_{1}(x)M_{p_{2},(5)}f_{2}(x).
\end{split}
\end{equation}
We  also can obtain
\begin{equation}
|m_{Q}(H_{3})|+|m_{Q}(H_{4})|\leq
C||b_{1}||_{\ast}M_{p_{1},(5)}f_{1}(x)M_{p_{2},(5)}f_{2}(x).
\end{equation}
For $H_{5}$, since $z\in Q$, by (i) of Definition 1.7, Lemma 2.1, Lemma
2.2, H\"{o}lder's inequality and $Q$ is a doubling ball, we deduce
\begin{equation}
\begin{split}
|H_{5}(z)|&\leq C
\int_{6^{N}B}\int_{X\backslash\frac{6}{5}Q}\frac{|b_{1}(y_{1})-m_{Q}b_{1}||f_{1}(y_{1})||f_{2}(y_{2})|d\mu(y_{1})d\mu(y_{2})}
{[\sum_{j=1}^{2}\lambda(z,d(x,y_{j}))]^{2}}\\
&\leq
C\int_{6^{N}B}|f_{2}(y_{2})|d\mu(y_{2})\sum_{k=1}^{\infty}\int_{6^{k}\frac{6}{5}Q}
\frac{|b_{1}(y_{1})-m_{Q}b_{1}||f_{1}(y_{1})|}{(\lambda(z,6^{k-1}\frac{6}{5}r_{Q}))^{2}}d\mu(y_{1})\\
&\leq
C\int_{6^{N}B}|f_{2}(y_{2})|d\mu(y_{2})\sum_{k=1}^{\infty}6^{-km}\\
&\ \ \ \times \int_{6^{k}\frac{6}{5}Q}
\frac{1}{\lambda(z,\frac{6}{5}r_{Q})}\frac{|b_{1}(y_{1})-m_{Q}b_{1}||f_{1}(y_{1})|d\mu(y_{1})}{\lambda(z,6^{k-1}\frac{6}{5}r_{Q})}\\
&\leq
C\frac{1}{\lambda(z,6r_{Q})}\int_{6^{N}B}|f_{2}(y_{2})|d\mu(y_{2})\sum_{k=1}^{\infty}6^{-km}\frac{1}{\lambda(z,5\times6^{k}\frac{6}{5}r_{Q})}\\
&\ \ \ \times
\biggl[\int_{6^{k}\frac{6}{5}Q}|b_{1}(y_{1})-m_{6^{k}\frac{6}{5}Q}(b_{1})||f_{1}(y_{1})|d\mu(y_{1})\\
&\ \ \ +\int_{6^{k}\frac{6}{5}Q}|m_{6^{k}\frac{6}{5}Q}(b_{1})-m_{Q}b_{1}||f_{1}(y_{1})|d\mu(y_{1})\biggr]\\
&\leq
C\frac{1}{\lambda(z,6r_{Q})}\int_{6^{N}B}|f_{2}(y_{2})|d\mu(y_{2})\sum_{k=1}^{\infty}6^{-km}\\
&\ \ \
\times\biggl[\biggl(\frac{1}{\lambda(z,5\times6^{k}\frac{6}{5}r_{Q})}\int_{6^{k}\frac{6}{5}Q}
|b_{1}(y_{1})-m_{6^{k}\frac{6}{5}Q}(b_{1})|^{p'_{1}}d\mu(y_{1})\biggr)^{1/p'_{1}}\\
&\ \ \ \ \times(\frac{1}{\lambda(z,5\times6^{k}\frac{6}{5}r_{Q})}\int_{6^{k}\frac{6}{5}Q}|f_{1}(y_{1})|^{p_{1}}d\mu(y_{1}))^{1/p_{1}}\\
&\ \ \ \ +k||b_{1}||_{\ast}\frac{1}{\lambda(z,5\times6^{k}\frac{6}{5}r_{Q})}\int_{6^{k}\frac{6}{5}Q}|f_{1}(y_{1})|d\mu(y_{1})\biggr]\\
&\leq
C\sum_{k=1}^{N}\frac{1}{\lambda(z,6r_{Q})}\int_{6^{k}B}|f_{2}(y_{2})|d\mu(y_{2})||b_{1}||_{\ast}M_{p_{1},(5)}f_{1}(x)\\
&\leq
C\sum_{k=1}^{N}\frac{\mu(5\times6^{k}B)}{\lambda(z,5\times6^{k}r_{B})}\frac{\lambda(z,5\times6^{k}r_{B})}{\lambda(z,6r_{Q})}\\
&\ \ \ \ \times\frac{1}{\mu(5\times6^{k}B)}\int_{6^{k}B}|f_{2}(y_{2})|d\mu(y_{2})||b_{1}||_{\ast}M_{p_{1},(5)}f_{1}(x)\\
&\leq
CK_{B,Q}||b_{1}||_{\ast}M_{p_{1},(5)}f_{1}(x)M_{p_{2},(5)}f_{2}(x).
\end{split}
\end{equation}
Then it yields
\begin{equation}
|m_{Q}(H_{5})|\leq
CK_{B,Q}||b_{1}||_{\ast}M_{p_{1},(5)}f_{1}(x)M_{p_{2},(5)}f_{2}(x).
\end{equation}
In the similar way to estimate $m_{Q}(H_{5})$, we also obtain
\begin{equation}
|m_{Q}(H_{6})|+|m_{Q}(H_{7})|\leq
CK_{B,Q}||b_{1}||_{\ast}M_{p_{1},(5)}f_{1}(x)M_{p_{2},(5)}f_{2}(x).
\end{equation}
From (2.8) in Lemma 2.2, we deduce
\begin{equation}
\begin{split}
F_{21}\leq&
CK_{B,Q}^{2}\{||b_{1}||_{\ast}||b_{2}||_{\ast}M_{r,(6)}(T(f_{1},f_{2}))(x)\\
&\ \ \ +||b_{1}||_{\ast}M_{r,(6)}([b_{2},T](f_{1},f_{2}))(x)\\
&\ \ \ +||b_{2}||_{\ast}M_{r,(6)}([b_{1},T](f_{1},f_{2}))(x)\\
&\ \ \ +||b_{1}||_{\ast}||b_{2}||_{\ast}M_{p_{1},(5)}f_{1}(x)M_{p_{2},(5)}f_{2}(x)\}.\\
\end{split}
\end{equation}
$F_{22}$ and $F_{23}$ also have similar estimate of $F_{21}$,
therefore,
\begin{equation}
\begin{split}
F_{2}\leq&
CK_{B,Q}^{2}\biggr\{||b_{1}||_{\ast}||b_{2}||_{\ast}M_{r,(6)}(T(f_{1},f_{2}))(x)\\
&\ \ \ +||b_{1}||_{\ast}M_{r,(6)}([b_{2},T](f_{1},f_{2}))(x)\\
&\ \ \ +||b_{2}||_{\ast}M_{r,(6)}([b_{1},T](f_{1},f_{2}))(x)\\
&\ \ \ +||b_{1}||_{\ast}||b_{2}||_{\ast}M_{p_{1},(5)}f_{1}(x)M_{p_{2},(5)}f_{2}(x)\biggr\}.
\end{split}
\end{equation}
From $F_{3}$ to $F_{6}$, using the similar method to estimate
$I_{4}$, we conclude

\begin{equation}
F_{3}+F_{4}+F_{5}+F_{6}\leq
C||b_{1}||_{\ast}||b_{2}||_{\ast}M_{p_{1},(5)}f_{1}(x)M_{p_{2},(5)}f_{2}(x).
\end{equation}
Thus (2.15) holds by from (2.27) to (2.40) and hence (2.11) is proved.  With
the same method to prove (2.11), we can obtain that (2.12) and (2.13) are
also hold. Here we omit the details. Thus the Lemma 2.5 is proved.
\end{proof}

\begin{proof}[Proof of Theorem 1.11]
 Let $0<\delta<1/2$, $1<p_{1},p_{2},q<\infty$,
$\dfrac{1}{q}=\dfrac{1}{p_{1}}+\dfrac{1}{p_{2}}$, $1<r<q$,
 $f_{1}\in
L^{p_{1}}(\mu)$, $f_{2}\in L^{p_{2}}(\mu)$, $b_{1}\in RBMO(\mu)$ and
$b_{2}\in RBMO(\mu)$.
 By $|f(x)|\leq
N_{\delta}f(x)$, Lemma 2.1, Lemma 2.3, Lemma 2.4, H\"{o}rder's inequality
and the boundedness of $M_{(\rho)}$ and $M_{r,(\rho)}$ for $\rho\geq
5$ and $q>r$, we obtain
\begin{equation}
\begin{split}
&||[b_{1},b_{2},T](f_{1},f_{2})||_{L^{q}(\mu)} \leq
||N_{\delta}([b_{1},b_{2},T](f_{1},f_{2}))||_{L^{q}(\mu)}\\
\leq &
C||M^{\sharp}_{\delta}([b_{1},b_{2},T](f_{1},f_{2}))||_{L^{q}(\mu)}\\
\leq &C||b_{1}||_{\ast}||b_{2}||_{\ast}||M_{r,(6)}(T(f_{1},f_{2}))||_{L^{q}(\mu)}\\
&
+C||b_{1}||_{\ast}||M_{r,(6)}([b_{2},T](f_{1},f_{2}))||_{L^{q}(\mu)}\\
&+C||b_{2}||_{\ast}||M_{r,(6)}([b_{1},T](f_{1},f_{2}))||_{L^{q}(\mu)}\\
&+C||b_{1}||_{\ast}||b_{2}||_{\ast}||M_{p_{1},(5)}f_{1}(x)M_{p_{2},(5)}f_{2}(x)||_{L^{q}(\mu)}\\
\leq
&C||b_{1}||_{\ast}||b_{2}||_{\ast}||f_{1}(x)||_{L^{p_{1}}(\mu)}||f_{2}(x)||_{L^{p_{2}}(\mu)}\\
& +C||b_{1}||_{\ast}||([b_{2},T](f_{1},f_{2}))||_{L^{q}(\mu)}\\
&+C||b_{2}||_{\ast}||([b_{1},T](f_{1},f_{2}))||_{L^{q}(\mu)}\\
\leq
&C||b_{1}||_{\ast}||b_{2}||_{\ast}||f_{1}(x)||_{L^{p_{1}}(\mu)}||f_{2}(x)||_{L^{p_{2}}(\mu)}\\
& +C||b_{1}||_{\ast}||M^{\sharp}_{\delta}([b_{2},T](f_{1},f_{2}))||_{L^{q}(\mu)}\\
&+C||b_{2}||_{\ast}||M^{\sharp}_{\delta}([b_{1},T](f_{1},f_{2}))||_{L^{q}(\mu)}\\
\leq &||b_{1}||_{\ast}||b_{2}||_{\ast}||f_{1}(x)||_{L^{p_{1}}(\mu)}||f_{2}(x)||_{L^{p_{2}}(\mu)}\\
&+C||b_{1}||_{\ast}||M_{r,(6)}(T(f_{1},f_{2}))(x)||_{L^{q}(\mu)}\\
&+C||b_{1}||_{\ast}||M_{p_{1},(5)}f_{1}(x)M_{p_{2},(5)}f_{2}(x)||_{L^{q}(\mu)}\\
&+C||b_{2}||_{\ast}||M_{r,(6)}(T(f_{1},f_{2}))(x)||_{L^{q}(\mu)}\\
&+C||b_{2}||_{\ast}||M_{p_{1},(5)}f_{1}(x)M_{p_{2},(5)}f_{2}(x)||_{L^{q}(\mu)}\\
\leq
&C||b_{1}||_{\ast}||b_{2}||_{\ast}||f_{1}(x)||_{L^{p_{1}}(\mu)}||f_{2}(x)||_{L^{p_{2}}(\mu)}.
\end{split}
\end{equation}
Thus the proof of Theorem 1.11 is completed.
\end{proof}

{\bf Acknowledgements} This work was supported by National Natural
Science Foundation of China (Grant No. 10371087) and Excellent Young Talent Foundation of Anhui Province (Grant No.2013SQRL080ZD).


\end{document}